\documentclass[11pt]{amsart}

\usepackage{a4wide}
\usepackage[T1]{fontenc}
\usepackage[latin1]{inputenc}
\usepackage{amssymb}
\usepackage{amsfonts}
\usepackage{latexsym}
\usepackage{psfrag}
\usepackage{epsfig}
\makeatletter
\newcommand{\cal}{\mathcal}

\newcommand{\<}{\langle}
\newcommand{\R}{\mathbb{R}}
\newcommand{\C}{\mathbb{C}}

\newcommand{\N}{\mathbb{N}}
\newcommand{\te}{\theta}

\newcommand{\al}{\alpha}
\newcommand{\be}{\beta}
\newcommand{\si}{\sigma}
\newcommand{\la}{\lambda}
\newcommand{\ep}{\varepsilon}
\newcommand{\Om}{\Omega}

\newcommand{\p}{\partial}
\newcommand{\ti}{\tilde}
\newcommand{\Ti}{\widetilde}

\newcommand{\To}{\longrightarrow}

\newcommand{\de}{\delta}

\newcommand{\Ga}{\Gamma}

\newcommand{\vphi}{\varphi}
\newcommand{\tr}{\mbox{tr}}
\newcommand{\im}{\mbox{Im}}
\newcommand{\re}{\mbox{Re}}

\numberwithin{equation}{section}

\def \.{{\bf{\cdot}}}

\def \im{{\mbox{Im }}}
\def \re{{\mbox{Re }}}

\def \hpe{{H_{0,\perp}}}
\def \hpa{{H_{0,\parallel}}}
\newtheorem{prop}{Proposition}[section]
\newtheorem{lem}{Lemma}[section]
\newtheorem{thm}{Theorem}[section]
\newtheorem{rem}{Remark}[section]

\newtheorem{defi}{Definition}[section]

\begin{document}

\makeatother
\title[
Magnetic schr\"odinger Resonances and SSF]
{Resonances and spectral shift function 
 for a 
\\magnetic
schr\"odinger operator \vspace{5mm}}
\author{Abdallah Khochman}
\date{\today}
\email{Abdallah.Khochman@math.u-bordeaux1.fr}
\address{ Universit\'e Bordeaux I, Institut de Math\'ematiques, UMR CNRS 5251, 351, cours de la Lib\'eration, 33405 Talence, France}
\maketitle{}

\begin{abstract}
We consider the 3D Schr\"odinger operator $H_0$ with 
constant magnetic field and subject to an electric potential $v_0$
depending only on the variable along the magnetic field $x_3$. The
operator $H_0$ has infinitely many eigenvalues of infinite
multiplicity embedded in its continuous spectrum. We perturb $H_0$ by smooth scalar potentials
$V=O(\<(x_1,x_2)\rangle^{-\de_\perp}\<x_3\rangle^{-\de_\parallel})$,
$\de_\perp>2,\ \de_\parallel>1$. We assume also that $V$ and $v_0$ have an analytic continuation, in
the magnetic field direction, in a complex sector outside a compact set.
We define the resonances of $H=H_0+V$ as the eigenvalues of the
non-selfadjoint operator obtained from $H$ by analytic 
 distortions of $\R_{x_3}$. 
 We study their distribution near any fixed real eigenvalue of $H_0$, $2bq+\la$ for $q\in\N$.
In a ring centered at $2bq+\la$ with radiuses $(r,2r)$, 
 we establish an 
upper bound, as $r$ tends to $0$, of the number of resonances. This
upper bound depends on the decay of $V$ at infinity only in the directions $(x_1,x_2)$. 
Finally, we deduce a representation of the derivative of the
spectral shift function (SSF) for the operator pair ($H_0,H$) in
terms of resonances.
 This representation justifies the Breit-Wigner
approximation and implies a local trace formula.
\end{abstract}

\maketitle{ {\it Mathematics classification}: 35P25, 35J10, 47F05,
81Q10.


Spectral shift function, Breit-Wigner approximation, Trace formula.
}

\section{Introduction}
The resonance theory for non-relativistic particles satisfying the
Schr\"odinger equation has been developed  following  several
approaches. Among them we can mention the analytic dilation (see
Aguilar-Combes \cite{JAJC}) or the analytic distortion (see Hunziker
\cite{WH}) and meromorphic continuation of the resolvent or of the
scattering matrix (see Lax-Philips \cite{PLRP} and Vainberg \cite{BRV}). For Schr\"odinger
operators with constant magnetic field, 
 the resonances can be defined by analytic dilation (only) with respect to the
variable along the magnetic field (see Avron-Herbst-Simon
\cite{JAIHBS}, Wang \cite{XPW},
Astaburuaga-Briet-Bruneau-Fern\'andez-Raikov \cite{MAPBVBCFGR}) 
and by meromorphic continuation of the resolvent (see
J.F.Bony-Bruneau-Raikov \cite{JBVBGR}).
\medskip

The link between the resonances and the spectral shift function
(SSF) by the so-called Breit-Wigner approximation has been developed
in different situations. Such a representation of the derivative of
the spectral shift function related to the resonances, implies trace
formulas. In the semi-classical regime we can mention Sj\"ostrand
\cite{JS1}, \cite{JS2}
, Petkov-Zworski
\cite{VPMZ}, J.F.Bony-Sj\"ostrand \cite{JBJS}, Bruneau-Petkov
\cite{VBVP} and Dimassi-Zerzeri \cite{MDMZ} for the Schr\"odinger
operator and \cite{AK} for the Dirac operator. In \cite{JBVBGR},
J.-F.Bony, Bruneau and Raikov obtain a Breit-Wigner approximation of
the
spectral shift function near a Landau level 
 for the $3$D Schr\"odinger operator with constant magnetic field.
 For the last operator, under more general assumptions, Fern\'andez-Raikov \cite{CFGR} studied the singularities of the spectral shift
 function at a Landau level. These singularities has been analysed
at eigenvalues of infinite multiplicity by Astaburuaga, Briet, Bruneau, Fern\'andez and Raikov
\cite{MAPBVBCFGR} for a 
 magnetic Schr\"odinger operator having electric potential
 depending (only) on the variable along the magnetic field.
\medskip

 In this paper we consider the 
magnetic
Schr\"odinger operator $H_0$ 
with an electromagnetic field introduced in \cite{MAPBVBCFGR}. We
suppose that the magnetic field is constant and that the electric
potential $v_0$ depends only on the variable $x_3$ and is analytic
outside a compact set. This operator is remarkable because of the
generic presence of infinitely many eigenvalues of infinite
multiplicity, embedded in the continuous spectrum of $H_0$. We
perturb the operator $H_0$ by a smooth scalar potential $V$ analytic
outside a compact set with respect to the variable $x_3$.
\medskip

The purpose of this work is to define the resonances of the
electromagnetic Schr\"odinger operator $H=H_0+V$ for analytic
perturbation outside a compact set in the third direction $x_3$. We
define the resonances for $H$ as the discrete eigenvalues of the
non-selfadjoint operator $H_\te$ obtained from the 
magnetic Schr\"odinger operator by a general class of complex
distortions of $\R_{x_3}$. 
In Section \ref{secdistortion}, we prove that the discrete
eigenvalues of $H_\te$ are the zeros of a regularized determinant
$\det_2(\cdot)$ which is independent of the distortion. This
justifies the definition of the resonances. We calculate the
essential spectrum of the distorted operator to determine the sector
where we can define the resonances. In Section \ref{secmajoration},
we establish an upper bound for the number of resonances of $H$ in a
domain of size $r\to0$ near
 an embedded eigenvalue of $H_0$. The
second goal of this work is to obtain a Breit-Wigner approximation
for the derivative of the spectral shift function
$\xi(\la)$ related to the resonances of the 
 operator $H$, as well as a local trace formula 
(see Section \ref{sectSSF}).
\section{Assumptions and results}\label{secassumption}
In this section, we summarize some 
spectral properties of the 3D Schr\"odinger operator $H_0$ with
constant magnetic field ${\bf B}=(0,0,b)$, $b>0$ and subject to a
non-constant electric field $E=-(0,0,v_0'(x_3))$ depending only on
the variable $x_3$ (see \cite{MAPBVBCFGR}). We also state the main
results. Let
\begin{eqnarray}
 H_0=\hpe\otimes I_{\parallel}+I_{\perp}\otimes \hpa,
\end{eqnarray}
where $I_\parallel$ and $I_\perp$ are the identity operators in $L^2(\R_{x_3})$ and $L^2(\R_{x_1,x_2}^2)$ respectively,
\begin{eqnarray}
 \hpe:=\left(i\frac{\p}{\p x_1}-\frac{bx_2}{2}\right)^2+\left(i\frac{\p}{\p x_2}+\frac{bx_1}{2}\right)^2 -b,\ \ \ (x_1,x_2)\in\R^2,
\end{eqnarray}
is the Landau Hamiltonian shifted by the constant $b$, self-adjoint
in $L^2(\R^2)$, and
\begin{eqnarray}
 \hpa:=-\frac{d^2}{dx_3^2}+v_0,\ \ \ x_3\in\R.
\end{eqnarray}
The operator $v_0$ is the multiplication operator by an one
dimensional scalar potential $v_0(x_3)$. We suppose that
$v_0\in L^\infty(\R)$
 and satisfies
\begin{eqnarray}\label{eq decroissance v0}
|v_0|=O(\langle x_3\rangle^{-\de_0}),
\end{eqnarray}
with $\<x\rangle=(1+|x|^2)^\frac{1}{2}$ and $\de_0>1$. Then using
Weyl theorem, we have
\begin{eqnarray}
 \si_{ess}(\hpa)=\si_{ess}(-\frac{d^2}{dx_3^2})=[0,+\infty[.
\end{eqnarray}
It is well known that the spectrum of the operator $\hpe$ consists
of the Landau levels $2bq$, $q\in\N:=\{0,1,2\dots\}$, and the
multiplicity of each eigenvalue $2qb$ is infinite (see
\cite{JAIHBS}).
 Consequently, 
 the eigenvalues of $H_0$ have the form $2bq+\la$ where $q\in\N$ and
 $\la$ is an eigenvalue of the one dimensional Schr\"odinger operator $H_{0,\parallel}=-\frac{d^2}{dx_3^2}+v_0(x_3)$.
 For simplicity, throughout the article we suppose also
 that
\begin{eqnarray}\label{eqinfspectre}
 \mbox{inf}\;\si(\hpa)>-2b.
\end{eqnarray}
Note that, (\ref{eqinfspectre}) holds true if $v_0>-2b$. The
eigenvalues of $H_0$, $2bq+\la,\ q\in\N^*$, are embedded in its
continuous spectrum
 $[0,+\infty[=\cup_{q=0}^\infty[2bq,\infty[$ and are
 of infinite multiplicity. 
\medskip


Now, we introduce the perturbed operator $H=H_0+V$ where $V$ is the
multiplication operator by the potential $V(x)$. Assume that
$V\in
L^\infty(\R^3)$ 
and satisfies
\begin{eqnarray}\label{eqdecroissance de V}
|V(x)|=O(\<X_\perp\rangle^{-\de_\perp}\<x_3\rangle^{-\de_\parallel}),
\ \ X_\perp=(x_1,x_2),
\end{eqnarray}
with 
$\de_\perp >2$ and $
\de_\parallel >1$.
 We suppose also that $V$ and $v_0$ have 
 holomorphic extensions in the magnetic field 
 direction $x_3$ in the sector
\begin{eqnarray}\label{secteurC0ep}
 C_{\epsilon,0}:=\{z\in\C; |\im(z)|\leq\epsilon|\re(z)|,\ \ |\re(z)|\geq R_0>0\},\ \ \mbox{for}\
 0<\epsilon<1,
\end{eqnarray} and
 satisfy respectively (\ref{eqdecroissance de V}) and (\ref{eq decroissance v0})
 for $x_3\in C_{\epsilon,0}$.
\medskip

For $\te\in D_\epsilon\cap\R$ with $D_\epsilon:=\{\te\in\C; \ |\te|\leq
r_\epsilon:=\frac{\epsilon}{\sqrt{1+\epsilon^2}}\}$, we denote
\begin{eqnarray*}
H_\te&:=&(I_{\perp}\otimes U_\te)H(I_{\perp}\otimes U_\te^{-1})=H_{0,\te}+V_\te,
\end{eqnarray*}
where
\begin{equation}\label{eqH0teta}H_{0,\te}:=(I_{\perp}\otimes U_\te)
H_0(I_{\perp}\otimes U_\te^{-1})=\hpe\otimes
I_{\parallel}+I_{\perp}\otimes \hpa(\te)
\end{equation}
 and  $\hpa(\te)=U_\te\hpa
U_\te^{-1}$ (see (\ref{definitionofUteta}) for the definition of $U_\te$).
 We will prouve in the next section that the operator $H_\te$ has an analytic extension for $\te\in D_\epsilon$.
\medskip

For $\te_0$ fixed in $D_\epsilon^+:=D_\epsilon\cap\{\te\in\C;\
\im(\te)\geq0\}$, $q\in\N$ and $r\in\R$, we define
\begin{eqnarray}
 \Ga_{r,\te_0}:=2br+(1+\te_0)^{-2}[0,+\infty[
\end{eqnarray}
and
\begin{eqnarray}
 S_{q,\te_0}:=
\bigcup_{q<r<q+1}\Ga_{r,\te_0}.
\end{eqnarray}
The spectrum of $H_{0,\te_0}$ is purely essential and we have
\begin{eqnarray}\label{eqspectreH0teta}
\si(H_{0,\te_0})=\si_{ess}(H_{0,\te_0})&=&\bigcup_{q\in\N}\left(2bq+\si(\hpa(\te_0))\right)
\\\nonumber &=&\bigcup_{q\in\N}\left(\Ga_{q,\te_0}\cup\left(2bq+\si_{disc}(\hpa(\te_0))
\right)\right),
\end{eqnarray}
where
$\si_{disc}(\hpa(\te_0))=\si_{disc}(\hpa)\cup\{z_1,z_2,\dots\}$,
$\si_{disc}(\hpa)$ denotes the discrete spectrum of $\hpa$ and
$z_1,z_2,\dots$ are the complex eigenvalues of $\hpa(\te_0)$. In the
following, we assume that $\si_{disc}(\hpa)=\{ \la \}$. Note that
$\la$ is necessarily simple.
\medskip

The essential spectrum of $H_\te$ coincides with that of
$H_{0,\te}$.
 We prove also that the discrete spectrum of $H_\te$ in $S_{\te}=\bigcup_{q\in \N} S_{q,\te}$
is independent of $\te$ in $D_\epsilon^+$ (i.e. for two values
$\te_1,\,\te_2$,
the discrete spectrum of $H_{\te_1}$ and $H_{\te_2}$ coincide on $S_{\te_1}\cap S_{\te_2}$), (see Proposition \ref{propresonance=zero}). 
This justifies the following definition.
\begin{defi}
 The resonances of $H$ in $S_{\te_0}$ are the discrete eigenvalues of $H_{\te_0}$. The multiplicity of a resonance $z_0$ is defined by
\begin{eqnarray}\label{def mult res}
 \emph{mult}(z_0):=\emph{rank}\frac{1}{2i\pi}\int_{\Ga_0}(z-H_{\te_0})^{-1}dz,
\end{eqnarray}
 where $\Ga_0$ is a small positively oriented circle centered at $z_0$. We will denote $\emph{Res}(H)$ the set of resonances.
\end{defi}

\begin{rem}
The resonances of $H$ in $\{z\in\C;\ \ \emph{Re}(z)<0\}$ are the
real discrete eigenvalues of $H$. 
\end{rem}
\begin{center}
{\large
\psfrag{S}[lt][][1][0]{$S_{\te_0}$}
\psfrag{*}[c][][1][0]{\begin{small}$\star $\end{small}}
\psfrag{0}[c][][1][0]{$0$} \psfrag{Omega}[c][][1][0]{$\Omega_q$}
\psfrag{la}[c][][1][0]{$\la$}
\psfrag{2bq}[c][][1][0]{$2b$}
\psfrag{2bq+la}[c][l][1][0]{$2b\!+\!\!\la$}
\psfrag{2b(q+1)}[c][ll][1][0]{$2bq$}
\psfrag{2b(q+1)+la}[c][0][1][0]{$2bq\!+\!\!\la$}
\psfrag{Re\(z\)}[c][][1][0]{$\mbox{Re}(z)$}
\psfrag{Im\(z\)}[c][][1][0]{$\mbox{Im}(z)$}
\psfrag{Ga0te}[lt][][1][0]{$\Gamma_{0,\te_0}$}
\psfrag{Ga1te}[lt][][1][0]{$\Gamma_{1,\te_0}$}
\psfrag{Gaqte}[lt][][1][0]{$\Gamma_{q,\te_0}$}
\psfrag{e}[lt][][1][0]{$\circ
$}
\psfrag{h}[lt][][1][0]{$\circ
$}
 \psfrag{e.g. of H}[lt][][1][0]{e.v.
of $H_{\te_0}$} \psfrag{titre}[c][][1][0]{Fig.1. The set
$S_{\te_0}$} \psfrag{Resonances of H}[l][][1][0]{Resonances of $H$}
\hspace{1.7cm}\includegraphics[height=9cm, width=15cm]{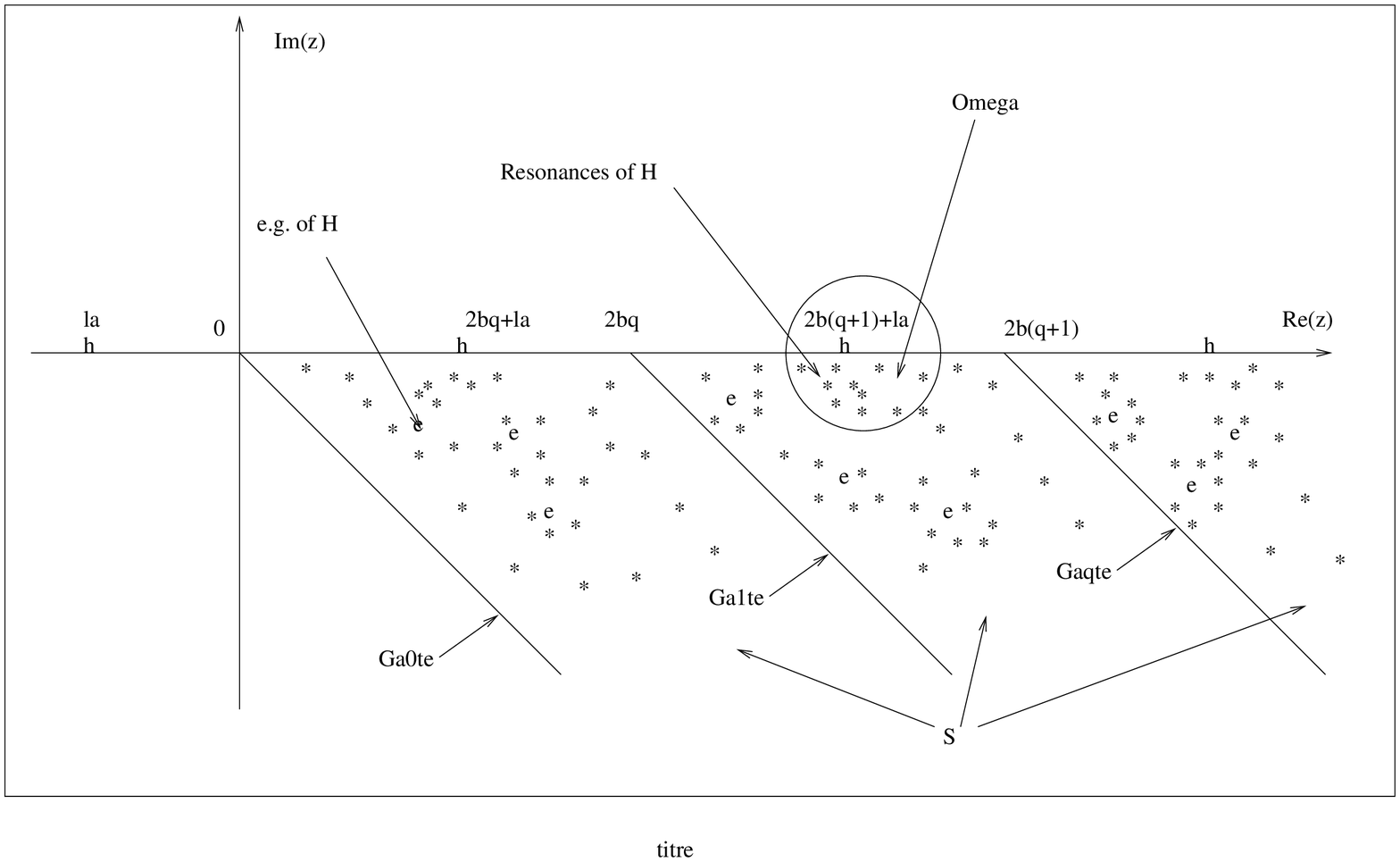} }
\end{center}
\vspace{5mm} Now, we state an upper bound as $r\to0$ on the number of
resonances of $H$ in a ring 
in $\Om_q$ with radiuses $(r,2r)$ and centered at $2bq+\la$,
$q\in\N^*$ fixed.
\begin{thm}\label{Thm upper bound}\emph{\bf[Upper bound]} Suppose that $V$ 
 and $v_0$ satisfy the above hypotheses. Then there
exist $r_0>0$ and $\nu>0$, such that, for any $0<r<r_0$,
\begin{equation}
\#\{z\in \emph{Res}(H)\cap\Om_q; r<|z-2bq-\la|<2r\}=O(n_+(r,\nu
p_qWp_q)|\ln r|),
\end{equation}
where $W=\sup_{x_3\in C_{\epsilon,0}}|\langle
x_3\rangle^{\de_\parallel} V|$, $p_q$ is the orthogonal projection
onto $\cal{H}_q:=\mbox{ker}(H_{0,\perp}-2bq)$ and $n_+(r,p_qWp_q)$
is the counting function of the eigenvalues larger than $r$ of the
Toeplitz operator $p_qWp_q$.  In particular, under our assumption we
have always $n_+(r,p_qWp_q)=O(r^{-2/\de_{\perp}})$.
\end{thm}
The counting function $n_+(r,p_qWp_q):=\mbox{rank }{\bf1}_{(r,+\infty)}(p_qWp_q)$ satisfies asymptotic
relations depending on the decay of $W$ at infinity.
 The following three
lemmas give an upper bound of $n_+(r,p_qWp_q)$ in the case
power-like decay, exponential decay, or compact support of $W$,
respectively. For more precise results concerning the asymptotic
properties, we refer to the cited theorem.

\begin{lem}\label{lenpqUpqdecroissane polonomyale}\emph{(Theorem 2.6 of \cite{GR})}
Let the function $U \in L^\infty(\R^2)
$ satisfy the estimate
$$
U(X_\perp)\leq C\langle X_\perp\rangle^{-\al},\ \ \ \ \ \ 
X_\perp\in\R^2,$$ for some $\al > 0$. 
 Then for each $q \in \N$, we have
$$n_+(r, p_qUp_q) 
=O(r^{-2/\al})
.$$
\end{lem}
\begin{lem}\label{lenpqUpqdecroissane exponentiel}\emph{(Theorem 2.1 of \cite{GRSW})} Let
$
U \in L^\infty(\R^2)$. Assume that
$$\limsup_{|X_\perp|\to\infty}\frac{\ln U(X_\perp)}{|X_\perp|^{2\be}}
<0,\ \ \ \ \  
X_\perp\in\R^2,$$ for some $\be>0
$ (with the convention $\ln(u)=-\infty$ if $u\leq0$). Then for each $q \in
\N$, we have
$$n_+(r, p_qUp_q) =O(
 \vphi_\be(r))$$
where, for $0<r<e^{-1}$,
\begin{equation}
 \vphi_\be(r):=\left\{
\begin{array}{ll}
|\ln r|^\frac1\be&\ \ \ \ \ \ \ \emph{if}\ \ \ \ 0<\be<1,\\
|\ln r|&\ \ \ \ \ \ \ \emph{if}\ \ \ \ \be=1,\\
(\ln|\ln r|)^{-1}|\ln r|&\ \ \ \ \ \ \ \emph{if}\ \
\ \ \be>1.
\end{array}
\right.
\end{equation}

\end{lem}
\begin{lem}\label{lenpqUpqsupportcompact} \emph{(Theorem 2.4 of \cite{GRSW})} Let $  U \in L^\infty(\R^2)$. Assume that the support of $U$ is compact. 
 Then for
each $q \in \N$, we have
$$n_+(r, p_qUp_q) = O(\vphi_\infty(r))
,$$
where, for $0<r<e^{-1}$,
$$\vphi_\infty(r) := (\ln | \ln r|)^{-1}|\ln r|.$$
\end{lem}

\medskip

Now, we study the spectral shift function (SSF) for the pair
$(H,H_0)$. The SSF $\xi(\la)$ for a pair of self-adjoint operators
($H,\,H_0$) is
a distribution in ${\cal D}'(\R)$
whose derivative is
\begin{equation}\label{eq def xi'}
\xi':f\in C_0^\infty(\R)\longmapsto-\tr\left(f(H)-f(H_0)\right).
\end{equation}
In our case, $|V|^\frac12(H_0+i)^{-1}$ 
is in the Hilbert-Schmidt class, (\ref{eq def
xi'}) is well defined and the SSF $\xi(\la)$ is a function in $L_{loc}^1(\R)$.\medskip

We will see further that the resonances of $H$ in $S_{\te_0}$ are
the zeros of the holomorphic extension of
$$z\in\{z\in\C,\;\im z>0\}\longmapsto D(z)=\mbox{det}_2((H-z)(H_0-z)^{-1})$$
into $S_{\te_0}$ 
 (see (\ref{eq de definition de det_2}) for the definition of $\det_2$).
Thus in order to obtain a link between the SSF and the resonances,
it will be convenient to introduce the regularized spectral shift
function
\begin{equation}\label{eq definition de xi2}
\xi_2(\nu)=\frac1\pi\lim_{\ep\to0^+}\arg
\mbox{det}_2\left((H-\nu-i\ep)(H_0-\nu-i\ep)^{-1}\right),
\end{equation}
whose derivative is the following distribution (see \cite{JBVBGR})
\begin{equation}\label{eqxi'2
dep} \xi_2':f\in
C_0^\infty(\R)\longmapsto-\tr\left(f(H)-f(H_0)-\frac{d}{d\ep}f(H_0+\ep
V)|_{\ep=0}\right).
\end{equation}
We will deduce the properties of the SSF from
those of the regularized SSF using the relation 
\begin{eqnarray}
\xi'=\xi_2'+\frac1\pi\mbox{Im tr}\left(
V_\te(H_{0,\te}-z)^{-2}\right).
\end{eqnarray}
%

We represent now, the derivative of the spectral shift function near
$2bq+\la$ as a sum of a harmonic measure related to the resonances
and the imaginary part of a holomorphic function.

Let $\Ti \Om\subset\subset\Om$ be open relatively compact subsets of
$\C\setminus\{0\}$.
 We assume that these sets are independent of $r$
and that $\Ti\Om$ is simply connected. Also assume that the
intersections between $\Ti\Om$ and $\R$ 
is a non-empty interval $I$. 
\begin{thm}\label{thm Breit-Wigner}\emph{\bf[Breit-Wigner approximation]}
We suppose $V$ and $v_0$ satisfy the above hypothesis.
 For
$\Ti\Om\subset\subset\Om$ and $I$ as above, there exists a function
$g$ holomorphic in $\Om$, such that for $\mu\in2bq+\la+rI$, we have
\begin{eqnarray*}
\xi'(\mu)=\frac{1}{\pi r} \emph{Im}\,
g'(\frac{\mu-2bq-\la}{r},r)-\hspace{-2mm}\sum_{\begin{array}{c}
\scriptstyle w\in \textmd{\emph{Res}}(H)\cap2bq+\la+r\Om \\
\scriptstyle \textmd{\emph{Im}}\,  w\neq0
\end{array}}\hspace{-2mm}\frac{-\emph{Im} w}{\pi|\mu-w|^2}-
\hspace{-2mm}\sum_{w\in
\textmd{\emph{Res}}(H)\cap2bq+\la+rI}\hspace{-2mm}\de(\mu-w)
\end{eqnarray*}
where $g(z,r)$ satisfies the estimate
\begin{equation}
g(z,r)=O\left(n_+(r,\nu p_qWp_q)|\ln r|+\Ti n_1(r/\nu)+\Ti
n_2(r/\nu)\right)=O(|\ln r|r^{-\frac{2}{\de_\perp}}),\ \ \ \nu>0,
\end{equation}
uniformly with respect to $0<r<r_0$ and $z\in \Ti\Om$, with $\Ti
n_p,\ p=1,\ 2,$ defined by 
\begin{equation}\label{eq de n tilde p}
\Ti
n_p(r):=\left\|\frac{p_qWp_q}{r}{\bf1}_{[0,r]}(p_qWp_q)\right\|_p^p,\
\ \ r>0.
\end{equation}
Here, $\|\cdot\|_p
$ stands for the trace-class norms $(p=1)$ and
Hilbert-Schmidt norms $(p=2)$.
\end{thm}
Using \cite[Corollary1]{JBVBGR}, for $W$ defined above satisfying
the assumption of Lemma \ref{lenpqUpqdecroissane polonomyale} with
$\al\geq2$, we have 
\begin{equation}
\Ti n_p(r)=O(r^{-\frac2\al}),\ p=1,2.
\end{equation}
 Finally, if
the assumption of Lemma \ref{lenpqUpqdecroissane exponentiel} or
\ref{lenpqUpqsupportcompact} hold for $W=U$, we have
\begin{equation}
\Ti n_p(r)=o(\vphi_\be(r))\ \ r\searrow0,
\end{equation}
the function $\vphi_\be(r) 
$ being defined in Lemma
\ref{lenpqUpqdecroissane exponentiel} or
\ref{lenpqUpqsupportcompact}.
\medskip

As in \cite{VPMZ}, \cite{VBVP} or \cite{JBVBGR} and repeating the
arguments used in the proof of \cite[Corollary 3]{JBVBGR}, we deduce
from Theorem \ref{Thm upper bound}-\ref{thm Breit-Wigner} the
following theorem
\begin{thm}\label{thm trace formula}\emph{\bf[Trace formula]}
Let $\Ti\Om\subset\subset\Om$ be as in Theorem \ref{thm
Breit-Wigner}. Suppose that $f$ is holomorphic on a neighborhood of
$\Om$ and that $\phi\in C_0^\infty(\Om\cap\R)$ satisfies $\phi=1$
near $\Ti\Om\cap\R$. Then, under the assumptions of Theorem \ref{thm
Breit-Wigner}, we have the following trace formula
\begin{equation*}\label{eq trace formula}
\emph{tr}\!\left(\!(\phi f)(\frac{H-2bq-\la}{r})-(\phi
f)(\frac{H_0-2bq-\la}{r})\!\right)\!=\hspace{-5mm}\sum_{w\in\textmd{\emph{Res}}(H)\cap2bq+\la+r\Ti\Om}\hspace{-8mm}f(\frac{w-2bq-\la}{r})+E_{f,\phi}(r)
\end{equation*}
with
$$|E_{f,\phi}(r)|\leq M_{\phi}\sup\{|f(z)|:z\in\Om\setminus \Ti\Om, \emph{Im}z\leq0\}\times N_q(r),$$
where $N_q(r)=n_+(r,\nu p_q Wp_q)|\ln r|+\Ti n_1(r/\nu)+\Ti
n_2(r/\nu)=O(|\ln r|r^{-\frac{2}{\de_\perp}}),$ and $M_{\phi}$
depends only on $\phi$.
\end{thm}
\section{Definition of resonances via
distortion analyticity}\label{secdistortion} In this section, we
start with the definition of the deformation for the electromagnetic
Schr\"odinger operator by analytic distortion on $\R_{x_3}$.
 We calculate the essential spectrum of the distorted Schr\"odinger operator
$H_\te$. 
 We prove that the discrete eigenvalues of 
 $H_\te$ are independent of the
distortion, this justifies the definition of a resonance as a
discrete eigenvalue of the distorted operator $H_\te$. We will also
prove that the resonances of $H$ repeated with their multiplicity
coincide with the zeros of a regularized determinant $\det_2(I+A)$
defined for a Hilbert-Schmidt operator $A$ by
\begin{eqnarray}\label{eq de definition de det_2}
 \mbox{det}_2(I+A):=\det((I+A)e^{-A}),
\end{eqnarray}
(see Krein \cite{MGK}). Let us now introduce the one-parameter
family of unitary distortions in the magnetic field direction $x_3$:
\begin{eqnarray}\label{definitionofUteta}
U_\te f(x)=J_{\phi_\te(x)}^{\frac12}f(\phi_\te(x)),\ \ \ \te\in\R, \
\ f\in S(\R),
\end{eqnarray}
where $\phi_\te(x)=x+\te g(x)$, $g:\R\longmapsto\R$ is a smooth
function and $J_{\phi_\te(x)}=\det(I+\te g'(x))$ is the Jacobian of
$\phi_\te(x)$. We suppose that $g$ satisfies the assumption
\begin{eqnarray*}
{\bf(A_g)}\left\{\begin{array}{l}
\mbox{(i)}\ \ \mbox{sup}_{x\in\R}|g'(x)|<1,\\
\mbox{(ii)}\ \,g(x)=0,\ \ \mbox{in the compact set }[-R_0,R_0],\ (\mbox{see }(\ref{secteurC0ep}) ),\\
\mbox{(iii)}\ g(x)=x,\ \ \mbox{outside a compact set }K(\supset
[-R_0,R_0]).
\end{array}\right.
\end{eqnarray*}
We recall that $$H_\te:=(I_{\perp}\otimes U_\te)H(I_{\perp}\otimes
U_\te^{-1})=H_{0,\te}+V_\te.$$ From (\ref{eqH0teta}) and using
Kato's theorem \cite[Theorem 4.5.35]{TK} we have the following (see
also Hunziker \cite{WH} for Schr\"odinger operator and \cite[Section
3]{AK} for the Dirac operator).
\begin{prop}\label{propH_te de type A} We suppose that the potential $V$ satisfies all the
assumptions of Section \ref{secassumption}. Then we have

\begin{itemize}
\item[(i)] $\te\in D_{\epsilon}\longmapsto H_{\theta}=H_{0,\theta}+
V_\te$ is an analytic family of type A.
\item [(ii)]$\si_{ess}(H_\te)=\si_{ess}(H_{0,\te})=2b\N+\si(\hpa(\te)).$
\end{itemize}
\end{prop}
\begin{lem}\label{lem spectre ess of H 0 pa teta}
The essential spectrum of $\hpa(\te)$ is
\begin{eqnarray}
\si_{ess}(\hpa(\te))=\Big\{\frac{\mu}{(1+\te)^2}\in\C;\ \ \
\mu\in[0,+\infty[\Big\}.
\end{eqnarray}
The rest of the spectrum is
$$ \si_{disc}(\hpa)\cup\{z_1,z_2,\dots\},$$
where $\si_{disc}(\hpa)$ denotes the discrete spectrum of $\hpa$ and
$z_1,z_2,\dots$ are the complex eigenvalues of $\hpa(\te)$.
\end{lem}
\begin{rem}
 The part $ \si_{disc}(\hpa)\cup\{z_1,z_2,\dots\}$ of the
 spectrum of $\hpa(\te)$ correspond to eigenvalues  with infinite multiplicity of $I_\perp\otimes \hpa(\te)$.
\end{rem}
\medskip

In the following we fix $q\in\N$ and a compact set $\Om_q$ centered
at
$2bq+\la$ such that \begin{equation}\label{definition of Omq}
\Om_q \cap \si_{ess}(H_\te)=\{2bq+\la\}.
\end{equation}

Repeating arguments in the proof of \cite[Proposition 1]{JBVBGR},
\cite[Lemma 1]{JBVBGR} and using the resolvent equation
$$(H_0-z)^{-1}=(H_0-v_0-z)^{-1}\left(I-v_0(H_0-z)^{-1}\right),$$
we obtain the following lemma.
\begin{lem}\label{lemresolvent hilbetschmidt}
 The operators $V(H_0-z)^{-1}$ and $\p_z(V(H_0-z)^{-1})$ are holomorphic  on $\{z\in\C;\ \emph{Im}(z)>0\}$ with values in the Hilbert-Schmidt class $S_2$ and in the trace class $S_1$ respectively.
\end{lem}

\begin{lem}\label{lemTVholomorphe z in Om_q}
The operators $V_\te(H_{0,\te}-z)^{-1}$ and
$\p_z(V_\te(H_{0,\te}-z)^{-1})$ are holomorphic for $z\in
\Om_q\backslash\{2bq+\la\}$ with values in the Hilbert-Schmidt class
$S_2$ and in the trace class $S_1$ respectively.
\end{lem}
\begin{proof}
 According 
  to (i) of Proposition \ref{propH_te de type A} and to the definition of $\Om_q$, the function
$z\longmapsto V_\te(H_{0,\te}-z)^{-1}$  is analytic for 
$z\in\Om_q\backslash\{2bq+\la\}$.
 Moreover, from the resolvent equation, for
 $z\in\Om_q\backslash\{2bq+\la\}$, we have
\begin{eqnarray}\label{eq resolv H_{0,te}}
V_\te(H_{0,\te}-z)^{-1}=V_\te(H_0-i)^{-1}\left(1+(H_0-H_{0,\te}+z-i)(H_{0,\te}-z)^{-1}\right).
\end{eqnarray}
If we denote by $p_q$ the orthogonal projection onto
$\cal{H}_q:=\mbox{ker}(H_{0,\perp}-2bq)$
%
 we have
\begin{eqnarray}
(H_{0,\te}-z)^{-1}=\sum_{q\in\N}p_q\otimes(H_{0,\parallel}(\te)+2bq-z)^{-1},\
\ z\in\Om_q\backslash\{2bq+\la\}.
\end{eqnarray}
Since
$$H_0-H_{0,\te}=I_\perp\otimes(H_{0,\parallel}-H_{0,\parallel}(\te))$$
the operator $(H_0-H_{0,\te})(H_{0,\te}-z)^{-1}$ is bounded.
Consequently, according to Lemma \ref{lemresolvent hilbetschmidt}
and equation (\ref{eq resolv H_{0,te}}), we obtain the lemma.
\end{proof}
Let us now introduce the function \begin{eqnarray}
z\in\Om_q\setminus\{2bq+\la\}\To
d_\te(z)=\mbox{det}_2(I+T_{V,\te}(z)),
\end{eqnarray}
with  $T_{V,\te}(z)=V_\te(H_{0,\te}-z)^{-1}$. The determinant
$d_\te(z)$ is well defined according to Lemma \ref{lemTVholomorphe z
in Om_q}
\begin{prop}\label{propresonance=zero}
Let $V$ and $v_0$ as in Section \ref{secassumption}. 
 The resonances of $H$ in $\Om_q$ are the zeros of the regularized determinant
$d_\te(z)=\emph{det}_2(I+T_{V,\te}(z))$ in
$\Om_q\backslash\{2bq+\la\}$,
 and are independent on
$\te\in D_\epsilon$ such that $\Om_q \cap
\si_{ess}(H_\te)=\{2bq+\la\}$.\medskip

 If $z_0$ is a resonance, there exists a
holomorphic function $f(z)$, for $z$ close to $z_0$, such that
$f(z_0)\neq0$ and
$$\emph{det}_2(I+T_{V,\te}(z))=(z-z_0)^{l(z_0)}f(z),$$
with $0<l(z_0)=\emph{mult}(z_0)$ where \emph{mult}$(z_0)$ is the multiplicity of the resonance defined by (\ref{def mult res}).
\end{prop}
\begin{proof}
Since the operator
$H_{0,\te}$ has no 
spectrum in $\Om_q\backslash\{2bq+\la\}$,
 we have
\begin{eqnarray}\label{eq resolv H_te}
 H_\te-z=\left(I+V_\te(H_{0,\te}-z)^{-1}\right)(H_{0,\te}-z).
\end{eqnarray}
Then, if $z\in\Om_q\backslash\{2bq+\la\}$ is a resonance of $H$
which is by definition a discrete eigenvalue of $H_\te$, the
determinant
$d_\te(z)=\mbox{det}_2(I+V_\te(H_{0,\te}-z)^{-1})=\mbox{det}_2(I+T_{V,\te}(z))$
vanishes. 
\medskip

Let us recall that, if $A$ is a bounded operator and B is a trace
class operator on some separable Hilbert space, we have
$\det(I+AB)=\det(I+BA)$. Moreover, for $A$ bounded and $B$
Hilbert-Schmidt, we have
\begin{eqnarray}\label{eqdet(1+AB)=det(1+BA)}
\mbox{det}_2(I+AB)=\mbox{det}_2(I+BA).
\end{eqnarray}
Then, the function $d_\te(z)=\det_2(I+T_{V,\te}(z))$ coincide with
$\det_2(I+T_{V,0}(z))=\mbox{det}_2(I+V(H_{0}-z)^{-1})$ for
$\te\in\R,\ \im z>0$ and by uniqueness of the extension, it is
independent on $\te\in D_\epsilon$. Since the resonances of $H$ are
the zero of $d_\te(z)$ the resonances are independents on $\te\in D_\epsilon$. 
\medskip

In a neighborhood of a zero $z_0$ of $d_\te(z)$ with multiplicity
$l(z_0)$, we write $d_\te(z)=(z-z_0)^{l(z_0)}G(z),$ where $G(z)$ is
a holomorphic function in a neighborhood of $z_0$ with
$G(z_0)\neq0$. Then, by the definition of $l_0(z)$,
$$l_0(z)=\frac{1}{2i\pi}\int_\Ga\p_z\ln\mbox{det}_2\left(1+T_{V,\te}(z)\right)dz,$$
where $\Ga$ is a small positively oriented circle centered at $z_0$. Further, we
have
\begin{eqnarray*}
 \p_z\ln\det\left(1+T(z)\right)=\tr\left((1+T(z))^{-1}\p_z T(z)\right),\ \ z\in\Om_q
\end{eqnarray*}
for any operator-valued holomorphic function $T(z)$ in the trace class $S_1$. Therefore,
\begin{eqnarray*}
\p_z\ln\mbox{det}_2\left(1+T_{V,\te}(z)\right)=\tr\left((1+T_{V,\te}(z))^{-1}\p_zT_{V,\te}(z)\right)-\tr\left(\p_zT_{V,\te}(z)\right).
\end{eqnarray*}
According to Lemma \ref{lemTVholomorphe z in Om_q},
$\p_zT_{V,\te}(z)$ is holomorphic in the trace class, then its
integral on $\Ga$ vanishes and (\ref{eq resolv H_te}) yields
\begin{equation*}
\begin{aligned}
l_0(z)&=\frac{1}{2i\pi}\int_\Ga\tr\left((H_\te-z)^{-1}V_\te(H_{0,\te}-z)^{-1}\right)dz\\
&=-\frac{1}{2i\pi}\int_\Ga\tr\left((H_\te-z)^{-1}-(H_{0,\te}-z)^{-1}\right)dz\\
&= \mbox{rank}\frac{1}{2i\pi}\int_\Ga\left((z-H_\te)^{-1}-(z-H_{0,\te})^{-1}\right)dz\\
&=\mbox{rank}\frac{1}{2i\pi}\int_\Ga(z-H_\te)^{-1}dz.
\end{aligned}
\end{equation*}
In the two latter equalities, we have used that the trace of the
projector coincide with its rank and the integral of
$(H_{0,\te}-z)^{-1}$ on $\Ga$ vanishes since it is holomorphic in
$\Om_q$.
\end{proof}

\section{Upper bound for the number of resonances near $2bq+\la$}\label{secmajoration}
In this section, we establish an upper bound on the number of
resonances in a ring of $\Om_q$ centered at $2bq+\la$ (see (\ref{definition of Omq})). 
 For $z\in\Om_q$, we write $z=2bq+\la+\eta$ where $\eta$ is a
complex number in a domain
centered at $0$ and  $0<r<|\eta|$
. Let 
$W=\sup_{x_3\in C_{\epsilon,0}}|\langle x_3\rangle^{\de_\parallel}
V|$. There exists a bounded function $M(x)$ such that
$$V(x)=W(X_\perp)\langle x_3\rangle^{-2\de_3}M(x)
, \;\mbox{ for }\;\de_3=\de_\parallel/2.$$

According to the previous section, the resonances in $\Om_q$ can be
identified with the points $z\in\Om_q$ where the determinant ${
d}_\te(z)=\det_2(I+T_{V,\te}(z))$ vanishes.
 Using (\ref{eqdet(1+AB)=det(1+BA)}), we have
 $$
d_\te(z)
=\mbox{det}_2(I+{\cal
 T}_{V,\te}(z))$$
with
\begin{equation}\label{eq cal TVteta}{\cal T}_{V,\te}(z)=
W^{\frac12}M_\te \langle
x_3\rangle_\te^{-\de_3}(H_{0,\te}-z)^{-1}W^{\frac12}\langle
x_3\rangle_\te^{-\de_3}
\end{equation}
 where
$\ M_\te=M(X_\perp,\phi_\te(x_3))$ 
and $\langle x_3\rangle_\te:=U_\te\langle x_3\rangle U_\te^{-1}=(1+(\phi_\te(x_3))^2)^\frac12$.
\bigskip

Using spectral theorem, for $\im z>0$, we can write
\begin{eqnarray}
(H_{0,\te}-z)^{-1}=\sum_{j\in\N}
p_j\otimes(H_{0,\parallel}(\te)-z+2bj)^{-1}.
\end{eqnarray}
In order to study the resonances near $2bq+\la$, we split ${\cal
T}_{V,\te}(z)$ into two parts:
\begin{eqnarray*}
 {\cal T}_{V,\te}(z)={\cal T}_{J,\te}^-+{\cal T}_{J,\te}^+,
\end{eqnarray*}
where ${\cal T}_{J,\te}^-(z)=\sum_{j\leq J}{\cal T}_{j,\te}$ and
${\cal T}_{J,\te}^+=\sum_{j>J}{\cal T}_{j,\te}$ for $J>q$
sufficiently large such that $\|{\cal T}_{J,\te}^+\|<\frac18$ and
$\|{\cal T}_{J,0}^+\|<\frac18$ (for that we use the
$h$-pseudo-differential calculus and the spectral theorem). Here,
$${\cal T}_{j,\te}=M_\te B_j\otimes
\<x_3\rangle_\te^{-\de_3}R_{j,\te}\<x_3\rangle_\te^{-\de_3},$$ with
$B_j=W^{\frac12}p_jW^{\frac12}$ and
$R_{j,\te}=(H_{0,\parallel}(\te)-z+2bj)^{-1}$. The operator
$\<x_3\rangle^{-\de_3}R_{j,0}\<x_3\rangle^{-\de_3}$ is of class
trace 
 (see \cite{MAPBVBCFGR}).
\bigskip

Further, let us decompose the self-adjoint operator $B_j$ into a
trace-class operator whose norm is bounded by $\ep/2$ for some
$\ep>0$ and an
operator of finite-rank 
 independent on $r$, namely
\begin{eqnarray}
B_j=B_j {\bf1}_{[0,\ep/2]}(B_j)+B_j{\bf1}_{]\ep/2,+\infty[}(B_j).
\end{eqnarray}
Then, for $j\neq q$, we have
\begin{eqnarray*}
{\cal T}_{j,\te}&=&M_\te B_j{\bf1}_{[0,\ep/2]}(B_j)\otimes
\<x_3\rangle_\te^{-\de_3}R_{j,\te}\<x_3\rangle_\te^{-\de_3}+M_\te
B_j{\bf1}_{]\ep/2,+\infty[}(B_j)\otimes
\<x_3\rangle_\te^{-\de_3}R_{j,\te}\<x_3\rangle_\te^{-\de_3}\\
&=&{\cal T}_{j,\te}^{<}+{\cal T}_{j,\te}^>.
\end{eqnarray*}

Let us now analyze the 
 term ${\cal T}_{q,\te}$. Denote by
$p_\parallel(\te)$ the spectral projection onto
Ker($H_{0,\parallel}(\te)-\la$). We have
$p_\parallel(\te)\cdot=\langle \cdot,\psi_{\bar\te}\rangle\psi_\te$
with $\psi_\te=U_\te^{-1}\psi$ 
and $\psi$ is an eigenfunction satisfying
$$H_{0,\parallel}\psi=\la\psi,\ \ \ \|\psi\|_{L^2(\R)}=1,\ \ \ \psi=\Bar\psi \mbox{ on }\R.$$
Then we have, for $\eta=z-2bq-\la$
\begin{equation*}
\begin{aligned}
{\cal T}_{q,\te}=& M_\te B_q\otimes
\<x_3\rangle_\te^{-\de_3}R_{q,\te}p_\parallel(\te)\<x_3\rangle_\te^{-\de_3}+M_\te
B_q\otimes
\<x_3\rangle_\te^{-\de_3}R_{q,\te}(I-p_\parallel(\te))\<x_3\rangle_\te^{-\de_3}
\\
=&-\frac1\eta\tau_q+ \ti{\cal T}_{q,\te},
\end{aligned}
\end{equation*}
with $\tau_q=M_\te B_q\otimes \<x_3\rangle_\te^{-\de_3}
p_\parallel(\te)\<x_3\rangle_\te^{-\de_3}$. We also have
\begin{eqnarray*}
\ti{\cal T}_{q,\te}&=&M_\te B_q{\bf1}_{[0,\ep/2]}(B_q)\otimes
\<x_3\rangle_\te^{-\de_3}R_{q,\te}(I-p_\parallel(\te))\<x_3\rangle_\te^{-\de_3}\\&
&+\;M_\te B_q{\bf1}_{]\ep/2,+\infty[}(B_q)\otimes
\<x_3\rangle_\te^{-\de_3}R_{q,\te}(I-p_\parallel(\te))\<x_3\rangle_\te^{-\de_3}\\
&=&\ti{\cal T}_{q,\te}^{<}+\ti{\cal T}_{q,\te}^>.
\end{eqnarray*}

 We denote by 
 $A^>(z)=\ti {\cal T}_{q,\te}^>+\sum_{j\neq q,\ j\leq J}{\cal
 T}_{j,\te}^>,$
and
$ A^<(z)=\ti {\cal T}_{q,\te}^<+\sum_{j\neq q,\ j\leq J}{\cal
 T}_{j,\te}^<+{\cal T}_{J,\te}^+.$
Then, we have \begin{equation}\label{eqdecompositioncalTV}
 {\cal T}_{V,\te}(z)=-\frac1\eta\tau_q+A^>(z)+A^<(z).
\end{equation}
We decompose the operator
$\tau_q$ into a trace-class operator whose norm is bounded by
$r\nu^{-1}$ for $\nu>0$ and an operator of finite-rank:
\begin{eqnarray*}\tau_{q}&=&M_\te B_q{\bf1}_{[0,\,r\nu^{-1}]}(B_q)\otimes
\<x_3\rangle_\te^{-\de_3}p_\parallel(\te)\<x_3\rangle_\te^{-\de_3}+M_\te
B_q{\bf1}_{]r\nu^{-1},+\infty[}(B_q)\otimes
\<x_3\rangle_\te^{-\de_3}p_\parallel(\te)\<x_3\rangle_\te^{-\de_3}\\&=&\tau_{q,1}+\tau_{q,2}.
\end{eqnarray*}
If we take $\ep$ sufficiently small and $\nu>0$ sufficiently large, we have the two following lemmas

\begin{lem}\label{lemcal T_V teta}
 Let $r_0>0$
 . For $z=2bq+\la+\eta\in \Om_q$ and $0<r<\emph{Im}\, \eta<r_0$, we have
\begin{equation}
 {\cal T}_{V,\te}(z)={\cal R}_\te(z)+ {\cal E}_\te(z),
\end{equation}
where the operator ${\cal R}_\te(z)\in S_1$ is the holomorphic operator defined by
\begin{equation}
{\cal R}_\te(z)=-\frac{1}{\eta}\tau_{q,2}+ A^>(z).
\end{equation}
The operator ${\cal E}_\te(z)\in S_2$ is the holomorphic operator
defined by
$${\cal E}_\te(z)= -\frac1\eta\tau_{q,1}+A^<(z).$$
Moreover, ${\cal E}_\te(z)$ satisfies the following estimate
\begin{equation}\label{eqestimationE_teta}
\|{\cal E}_\te(z)\|<\frac34.\end{equation}
\end{lem}
Using the limiting absorption principle for
$\<x_3\rangle^{-\de_3}R_{j,0}\<x_3\rangle^{-\de_3},\ j<q$, the
decomposition (\ref{eqdecompositioncalTV}) is also available for
$\te=0$, and 
 we have
\begin{lem}\label{lemcal T_V zero}
For $z=2bq+\la+\eta\in\Om_q$ and $0<r<\emph{Im}\,\eta<r_0$, we have
\begin{equation}
 {\cal T}_{V,0}(z)={\cal R}_0(z)+ {\cal E}_0(z),
\end{equation}
with ${\cal R}_0(z)={\cal R}_\al(z)\Big|_{\al=0}$ and $\;{\cal
E}_0(z)={\cal E}_\al(z)\Big|_{\al=0}$.
 Moreover, ${\cal E}_0(z)$ satisfies the following estimate
\begin{equation}\label{eqestimationE_zero}
\|{\cal E}_0(z)\|<\frac34.\end{equation}
\end{lem}
\begin{prop}\label{prop tild D teta coincide}
 Let $V$ and $v_0$ as in Section \ref{secassumption}. For $0<r<|\eta|<r_0$ with $r_0$ sufficiently small, $z=2bq+\la+\eta\in\Om_q$
 is a resonance of $H$ if and only if $z$ is a zero of
\begin{equation}\label{eqD(z,s)}
 {\cal D_\te(z,r)}=\det\left(I+{\cal R}_\te(z)(I+{\cal E}_\te(z))^{-1}\right),
\end{equation}
where ${\cal R}_\te(z)$ is a class trace operator.
 Moreover, for $\emph{Im } z>0$, the determinant
${\cal D}_\te(z,s)$ coincides with
$${\cal D}_0(z,s)=\det\left(I+{\cal R}_0(z)(I+{\cal
E}_0(z))^{-1}\right).$$

\end{prop}
\begin{proof}
%
By Proposition \ref{propresonance=zero}
, for $r<|\eta|<r_0$, $z$ is a resonance of $H$ if and only if $z$
is a zero of $ d_\te(z)=\det_2(I+{\cal R}_\te(z)+{\cal E}_\te(z))$.
We can write
$$ d_\te(z)=\det(I+{\cal R}_\te(z)(I+{\cal E}_\te(z))^{-1})\det((I+{\cal E}_\te(z))e^{-{\cal T}_{V,\te}(z)}).$$
According to (\ref{eqestimationE_teta}), we have $\det((I+{\cal
E}_\te(z))e^{-{\cal T}_{V,\te}(z)})\neq0,$ and then the zeros of $
d_\te(z)$ are the zeros of ${\cal D}_\te(z,r)$ with the same
multiplicity.
\bigskip

Using the theory of $h$-pseudo-differential operators (see
\cite{MDJS}), the resolvent $(H_{0,\parallel}(\al)-z+2bj)^{-1}$ is
uniformly bounded for $\al\in D_\epsilon^+$, $j\leq J$ and $\im z>0$
sufficiently large.
Then for $z$ fixed with $\mbox{Im}z\gg 1$, $\te\to {\cal
D}_\te(z,r)$ is a holomorphic function on $D_\epsilon^+$ (since the
construction of ${\cal E}_\te(z)$ is not uniform with respect to
$\te$, this property is not clear for $\im z>0$ near the real axis).
 Using that for  $\te\in\R$
$${\cal D}_\te(z,r)=\det\left(I+U_\te {\cal R}_0(z)(I+{\cal
E}_0(z))^{-1}U_\te^{-1}\right)=\det\left(I+{\cal R}_0(z)(I+{\cal
E}_0(z))^{-1}\right),$$ the function $\te\longmapsto{\cal
D}_\te(z,r)$ is constant on the real axis. Thus, by uniqueness of
the extension on $\te$, the determinant ${\cal D}_\te(z,r)$
coincides with ${\cal D}_0(z,r)$ for $\im z\gg1$
and $\te\in D_\epsilon^+ $. 
Moreover, since for $\te$ fixed in $D_{\epsilon}^{+}$,
$z\longmapsto{\cal D}_\te(z,r)$ and $z\longmapsto{\cal D}_0(z,r)$
are well defined and holomorphic for $\im z>0$ (see Lemmas \ref{lemcal T_V teta}, \ref{lemcal T_V zero}), ${\cal D}_\te(z,r)$ coincides with ${\cal D}_0(z,r)$ for $\im z>0$.
\end{proof}

Since there exists an operator $C:\ L^2(\R^2)\to L^2(\R^2)$ such
that $B_q=C^*C$ and $CC^*=p_qW(X_\perp)p_q$ (see \cite{CFGR} and
\cite{JBVBGR}), then for any $r>0$ we have
\begin{equation}\label{eqn_+Bq=n_+pqWpq}
 n_+(r,B_q)=n_+(r,p_qWp_q),
\end{equation}
where, for a compact self-adjoint operator $A$ and $r>0$, we set
$n_+(r,A)=\mbox{rank\,}{\bf1}_{(r,+\infty)}(A)$.
\begin{lem}\label{lemmajorationdeterminant}
For $z=2bq+\la+\eta\in\Om_q$ and $0<r<|\eta|<r_0$, there exists $\nu>0$ such that 
\begin{equation}\label{eq minoration of Ti D teta (z,s)}
{\cal D}_\te(z,r)=O(1)\exp\left(O(n_+(r,\nu p_qWp_q)+1)|\ln
r|\right).
\end{equation}
\end{lem}
\begin{proof}
%
Since $z\to A^>(z)$ is holomorphic near $z=2bq+\la$ or $\eta=0$ with
values in $S_1$, for $r_0$ sufficiently small, there exist a
finite-rank operator $A_0^>$ independent of $z$ and $\ti A^>(z)$
holomorphic in $S_1$ near $z=2bq+\la$ with $\|\ti
A^>(z)\|_{\rm{tr}}\leq\frac18$, $|\eta|\leq r_0$ such that
\begin{equation}\label{eqdecompositioncalA>}
A^>(z)=A_0^>+\ti A^>(z).
\end{equation} Since we have $\|\ti
A^>(z)\|_{\rm{tr}}\leq\frac18$, for
$0<r<|\eta|<r_0$,
$$\det\Big(I+\ti A^>(z)(I+{\cal E}_\te)^{-1}\Big)\neq0.$$

It follows that for $0<r<|\eta|<r_0$, the zeros of ${\cal
D}_\te(z,r)$ are the zeros of
\begin{equation}\label{eqD(z,s)}
 D_\te(z,r)=\det\left(I+K_\te(z,r)\right),
\end{equation}
 with
\begin{eqnarray}
K_\te(z,r)=\big(-\frac1\eta\tau_{q,2}+A_0^>\big)\Big(I+{\cal
E}_\te+\ti A^>(z)\Big)^{-1}.
\end{eqnarray}
We recall that $\tau_{q,2}=M_\te
B_q{\bf1}_{]r\nu^{-1},+\infty[}(B_q)\otimes
\<x_3\rangle_\te^{-\de_3}p_\parallel(\te)\<x_3\rangle_\te^{-\de_3}$.
Since the rank of the projector $p_\parallel(\te)$ is equal to $1$,
the rank of the operator $K_\te(z,r)$ is bounded by
$O(n_+(r\nu^{-1},B_q)+1)=O(n_+(r,\nu p_qWp_q)+1)$ (see
(\ref{eqn_+Bq=n_+pqWpq})) and its norm is bounded by
$O(|\eta|^{-1})=O(r^{-1})$ (see also Proposition \ref{prop tild D
teta coincide}).
\medskip

By the properties of $K_\te(z,r)$ for $0<r<|\eta|=|z-2bq-\la|<r_0$,
we have
\begin{equation}
D_\te(z,r)=\prod_{j=1}^{O(n_+(r,\,\nu
p_qWp_q)+1)}(1+\la_j(z,r))=O(1)\exp\left(O(n_+(r,\nu p_qWp_q)+1)|\ln
r|\right),
\end{equation}
uniformly with respect to $(z,r)$, where $\la_j(z,r)$ are the
eigenvalues of $K_\te(z,r)$ which satisfy $\la_j(z,r)=O(|r|^{-1})$.
Since
\begin{equation}\label{eq TiD=D*det}
{\cal D}_\te(z,r)=D_\te(z,r)\det\Big(I+\ti A^>(z)(I+{\cal
E}_\te)^{-1}\Big),
\end{equation}
and the norm of $\det\Big(I+\ti A^>(z)(I+{\cal E}_\te)^{-1}\Big)$ is
uniformly  bounded, the lemma follows.
\end{proof}
\begin{lem}
 For
$z=2bq+\la+\eta\in\Om_q$ , and $0<r<\emph{Im }\eta<r_0$, there exists $\nu>0$ such that 
\begin{equation}\label{eq majoration of Ti D teta (z,s)}
 |{\cal D}_0(z,r)|\geq C\exp\left(-C(n_+(r,\nu p_qWp_q)+1)|\ln r|\right),
\end{equation}
uniformly with respect to $(z,r)$.
\end{lem}
\begin{proof}
Repeating the argument (\ref{eqdecompositioncalA>}) in the proof of
Lemma \ref{lemmajorationdeterminant} for $\te=0$ and using Lemma
\ref{lemcal T_V zero}, 
 there
exist a finite-rank operator $K_0(z,r)$ satisfying
$$\mbox{rank }K_0(z,r)=O\left(n_+(r,\nu p_qWp_q)+1\right),\ \ \ \ \ \|K_0(z,r)\|=O(r^{-1}),$$
uniformly with respect to $r<|\eta|<r_0$ and an operator
$\varepsilon(z)$ such that
$${\cal
D}_0(z,r)=\det\left(I+K_0(z,r)\right) 
\det(I+\varepsilon(z))$$ with
$\|\varepsilon(z)\|_{\rm{tr}}\leq\frac34$ (see (\ref{eq
TiD=D*det})).
\\\\
Let us now estimate 
$D_0(z,r)^{-1}= \det\left(I+K_0(z,r)\right)^{-1}$. For
$\mbox{Im}z>r$ , we have
\begin{equation}\label{eq D_0(z,s)-1}D_0(z,r)^{-1}=\det\left((I+K_0)^{-1}
\right)=\det\left(I-K_0(I+K_0)^{-1}\right).\end{equation}
\\
By the construction of $K_0$, it satisfies 
$$I+K_0=\Big(I+{\cal T}_{V,0}(z)\Big)\Big(I+\Ti{\cal E}_0\Big)^{-1},$$
with $\Ti{\cal E}_0$ an operator bounded as $\|\Ti{\cal
E}_0\|<\frac78$ and
$${\cal T}_{V,0}(z)={\cal T}_{V,\te}(z)\Big|_{\te=0}=W^\frac{1}{2}\langle
x_3\rangle^{-\de_3}M(H_0-z)^{-1}\langle
x_3\rangle^{-\de_3}W^{\frac12}.$$ Using the resolvent equation, the
operator $I+{\cal T}_{V,0}(z)$ is invertible for $\im z>r>0,$ and
$$\Big(I+{\cal T}_{V,0}(z)\Big)^{-1}=I-W^\frac{1}{2}\langle
x_3\rangle^{-\de_3}M(H-z)^{-1}\langle
x_3\rangle^{-\de_3}W^{\frac12}.$$ Then $I+K_0$ is invertible for
$\im z>r$ and from the spectral theorem 
\begin{equation}\label{eq (I+K0)-1}\|(I+K_0)^{-1}\|=O(1+\|W^\frac{1}{2}\langle
x_3\rangle^{-\de_3}M(H-z)^{-1}\langle
x_3\rangle^{-\de_3}W^{\frac12}\|)=O(1+\frac{1}{|\im
z|}).\end{equation} Since the operator $K_0$ is of finite-rank
$O(n_+(r,\nu p_qWp_q)+1)$ and using (\ref{eq D_0(z,s)-1}) and
(\ref{eq (I+K0)-1}), we obtain the lemma.
\end{proof}

 The following lemma contains a version of the well known Jensen
inequality which is suitable for our purposes (see \cite{JBVBGR} for
the proof).
\begin{lem}
 Let $\Om$
 be a simply connected domain of $\C$ and let $g$ be a holomorphic
  function in $\Om$ with continuous extension to $\Bar\Om$. Assume there exists
 $z_0
 \in\Om$
 such that $g(z_0)\neq 0$ and $g(z)\neq0$ for $z\in\p\Om$.
  Let $z_1,z_2,\dots,\,z_N\in\Om$ be the zeros of $g$ repeated according to their multiplicity.
   For any domain $\Om'\subset\subset\Om$, there exists $C>0$ such that $N(\Om',g)$, the number of zeros $z_j$ of $g$ contained in $\Om'$,
 satisfies
$$N(\Om',g)\leq C\left(\int_{\p\Om}|\ln|g(z)||dz+|\ln|g(z_0)|| \right).$$
\end{lem}
Now, applying this lemma to the function $g(z)={\cal D}_\te(z,r)$,
we deduce from (\ref{eq minoration of Ti D teta (z,s)}), (\ref{eq
majoration of Ti D teta (z,s)}) and Proposition \ref{prop tild D
teta coincide} the upper bound on the number of resonances near
$2bq+\la$ stated in Theorem \ref{Thm upper bound}.

\section{Spectral shift function and resonances}\label{sectSSF}
In this section, we represent the derivative of the spectral shift
function (SSF) near $2bq+\la$ for $q\in\N$ as a sum of a
harmonic measure related to resonances, and the imaginary part of a
holomorphic function. As in \cite{VPMZ}, \cite{VBVP}, \cite{MDMZ},
and \cite{JBVBGR} such representation justifies the Breit-Wigner
approximation and implies a trace formula. We deduce also an
asymptotic expansion of the SSF near $2bq+\la$; in the case of
$v_0=0$, this expansion is given in \cite{JBVBGR}.
For a positive potentials $V$ which decay slowly enough as
$\|X_\perp\|\to \infty$, this expansion yields a remainder estimate
for the
corresponding asymptotic relations obtained in \cite{CFGR}.\\


In order to obtain such a representation formula, the first step is
the factorization of the generalized perturbation determinant. To
this end, we need some complex-analysis results due to Sj\"ostrand,
summarized in the following
\begin{prop}\label{propomegatilde}(see \cite{JS1}, \cite{JS2}) Let
$\Om$ be an open simply connected domain 
in $\C\setminus\{0\}$ 
 such that $\Om\cap\R
 $ is an interval. Let $z\To
 F(z,h), \ 0<h<h_0$, be a family of holomorphic functions in 
 $\Om$ containing
 a number $N(h)$ of zeros. 
We suppose that,\[F(z,h)=O(1)e^{O(1)N(h)},\;\;z\in\Om,\] and for all
$\rho>0$  small enough,  there exists  $C>0$ such that for all
$z\in\Om_{\rho}:=\Om\cap\{\emph{Im} \,z>\rho
\}$ we
have \[|F(z,h)|\geq e^{-CN(h)}.\]
 Then for each open simply connected subset  $\ti\Om\Subset\Om$ there exists $g(.,h)$
 holomorphic in $\ti\Om$ such that
 \[F(z,h)=\prod_{j=1}^{N(h)}(z-z_j)e^{g(z,h)},\;\;\p_zg(z,h)=O(N(h)),\;\;z\in\ti\Om.\]
\end{prop}
Let $\Ti\Om\subset\subset\Om$ be open relatively compact subset of
$\C\setminus\{0\}$.
 We assume that these sets are
independent of $r$ and that $\Ti\Om$ is simply connected. Also
assume
that the intersections between $\Ti\Om$ and $\R$ 
is a non empty interval $I$. 
 With
these hypotheses we can obtain the following representation of the
regularized spectral shift function near
$2bq+\la$.
\begin{thm}\label{thm representation formula}\emph{\bf[Representation formula]}
Suppose that $V$ and $v_0$ satisfy the hypotheses of Section
\ref{secassumption}. For $\Ti\Om\subset\subset\Om$ and $I$ as above,
there exists a function $g$ holomorphic in $\Om$, such that for
$\mu\in2bq+\la+rI$, we have
\begin{eqnarray}
\xi_2'(\mu)\!\!&=&\!\!\frac{1}{\pi r}
\emph{Im}g'(\frac{\mu-2bq-\la}{r},r)-\hspace{-4mm}\sum_{\begin{array}{c}
\scriptstyle w\in \textmd{\emph{Res}}(H)\cap2bq+\la+r\Om \\
\scriptstyle \textmd{\emph{Im}}\,  w\neq0
\end{array}}\hspace{-4mm}\frac{-\emph{Im} w}{\pi|\mu-w|^2}-
\hspace{-2mm}\sum_{w\in \textmd{\emph{Res}}(H)\cap2bq+\la+rI}\hspace{-4mm}\de(\mu-w) \nonumber\\
&&-\frac1\pi\emph{Im tr}\left(\p_z{\cal T}_{V,\te}(\mu)\right),
\end{eqnarray}
where $g(z,r)$ satisfies the estimate
\begin{equation}
g(z,r)=O\left(n_+(r,\nu p_qWp_q)|\ln r|+\Ti n_1(r/\nu)+\Ti
n_2(r/\nu)\right)=O(|\ln r|r^{-\frac{2}{\de_\perp}})
\end{equation}
uniformly with respect to $0<r<r_0$ and $z\in \Ti\Om$, with $\Ti
n_p,\ p=1,2,$ defined by (\ref{eq de n tilde p}).
\end{thm}
\begin{proof}
First, using the resolvent equation, we have
$$\mbox{det}_2\left((H-z)(H_0-z)^{-1}\right)=\mbox{det}_2\left(I+T_{V,0}(z)\right).$$
Using (\ref{eqdet(1+AB)=det(1+BA)}), the last determinant coincides
with $ d_\te(z)=\mbox{det}_2(I+{\cal T}_{V,\te}(z))$ for
$\te\in\R$, where ${\cal T}_{V,\te}(z)$ is defined in (\ref{eq cal
TVteta}). According to previous section, ${\cal T}_{V,\te}(z)$ is
extended on $\te\in D_\epsilon$ and
$$ d_\te(z)=D_\te(z,r)\det\left((I+\ti A^>(z)+{\cal E}_\te(z))e^{-{\cal T}_{V,\te}(z)}\right)$$
where $D_\te(z,r)$ is defined by (\ref{eqD(z,s)}).

By the properties of $\ti A^>(z)$ (see
(\ref{eqdecompositioncalA>})), for $\Ti K(z)=\ti A^>(z)+{\cal
E}_\te(z)$, the difference ${\cal T}_{V,\te}(z)-\Ti
K(z)=-\frac{1}{\eta}\tau_{q,2}+A_0^>$ is a finite-rank operator.
Using the fact that $\det_2(I+B)=\det(I+B)e^{-\tr B}$ for a
trace-class operator $B$, we have
\begin{equation}
\det\left((I+\ti A^>(z)+{\cal E}_\te(z))e^{-{\cal
T}_{V,\te}(z)}\right)=\mbox{det}_2(I+\Ti K(z))e^{-\tr({\cal
T}_{V,\te}(z)-\Ti K(z))},
\end{equation}
where $\det_2(I+\Ti K(z))$ is a non-vanishing holomorphic function.
Since $\ti A^>(z)$ is holomorphic in $S_2$ and
$$\|\frac{B_q}{r}{\bf1}_{[0,r]}(B_q)\|_2^2=-\int_0^r\frac{u^2}{r^2}dn_+(u,B_q)=\Ti n_2(r),$$
we have
$$\|\Ti K(z)\|_2^2=O(\Ti n_2(r/\nu)),$$
which implies that $|\det_2(I+\Ti K(z))|=O(\exp(\Ti n_2(r/\nu)))$.
Using moreover that $\|\Ti K(z)\|<1$, we have also $|\det(I+\Ti
K(z))|^{-1}=O(\exp(\Ti n_+(r/\nu)))$. Then there exists
$g_1(\cdot,r)$ holomorphic on $\Om$ such that,
$\frac{d}{dz}g_1(z,r)=O(\Ti n_2(r/\nu))$, on $\Ti\Om$, and
$$\mbox{det}_2(I+\Ti K(z))=e^{g_1(z,r)}.$$

We consider now the functions
$$F_{\te}:\ z\in\Om\longmapsto D_\te(z,r).$$
 The functions $F_{\te}$ are holomorphic in $\Om$ and
 $\Ti w\in\Om$ is a zero of $F_{\te}$ if
and only if $z=2bd+\la+\Ti wr$ is a resonance of $H$. Then applying
Proposition \ref{propomegatilde} to $F=F_{\te}$ with $h=r$,
$N(r)=n_+(r,\nu p_qWp_q)|\ln r|$, we obtain existence of functions
$g_{0}$ holomorphic in $\Om$ such that for $z\in\Om$, we have the
following factorization:
\begin{equation}
F_{\te}=\prod_{w\in\textmd{\mbox{Res}}(H)\cap2bq+\la+r\Om}\left(\frac{zr+2bq+\la-w}{r}\right)e^{g_{0}(z,r)},
\end{equation}
with
\begin{equation}
\frac{d}{dz}g_{0}(z,r)=O(n_+(r,\nu p_qWp_q)|\ln r|),
\end{equation}
uniformly with respect to $z\in \Ti\Om$.

Then by definition of $\xi_2$ (see (\ref{eq definition de xi2})),
for $\mu\in2bq+\la+r(\Om\cap\R)$ we obtain
\begin{eqnarray*}
\xi_2'(\mu)\!\!\!&=&\!\!\!\frac{1}{\pi
r}\mbox{Im}\p_z(g_{0}+g_1)(\frac{\mu-2bq-\la}{r},r)-\hspace{-8mm}\sum_{\begin{array}{c}
\scriptstyle w\in \textmd{\mbox{Res}}(H)\cap2bq+\la+r\Om \\
\scriptstyle \textmd{\mbox{Im}}\,  w\neq0
\end{array}}\hspace{-6mm}\frac{-\im w}{\pi|\mu-w|^2}-\hspace{-4mm}\sum_{w\in \textmd{\mbox{Res}}(H)\cap2bq+\la+rI}\hspace{-5mm}\de(\mu-w)\\
&&+\frac{1}{\pi r}\im\tr\left(\p_z\Ti
K(\frac{\mu-2bq-\la}{r})\right)-\frac1\pi\im\tr\left(\p_z{\cal
T}_{V,\te}(\mu)\right).
\end{eqnarray*}
Then, we conclude the proof of Theorem \ref{thm representation
formula} with $g=g_{0}+g_1+g_2$ taking
$$g_2(z,r)=\tr(\Ti
K(z)),$$
which satisfies $\frac{d}{dz}g_2(z,r)=O(\Ti
n_1(r/\nu))$.
\end{proof}
\begin{lem}\label{lemma relation xi et xi2}
 On
$\R\setminus(\{2b\N+\la\}\cup\{2b\N\})$, for $\te\in D_\epsilon^+,\
\emph{Im }\te>0$, we have
\begin{eqnarray}
\xi'=\xi_2'+\frac1\pi\emph{Im tr}\left(\p_z {
T}_{V,\te}(\cdot)\right),
\end{eqnarray}
where $T_{V,\te}(z)=V_\te(H_{0,\te}-z)^{-1}$.
\end{lem}
\begin{proof}
We follow the proof of \cite[Lemma 8]{JBVBGR}. From (\ref{eqxi'2
dep}), we have only to prove
$$\tr\left(\frac{d}{d\ep}f(H_0+\ep
V)|_{\ep=0}\right)=-\frac1\pi\int_\R f(\rho)\,\mbox{Im tr}\left(\p_z
{
 T}_{V,\te}(\rho)\right) d\rho,$$ for any $f\in
C_0^\infty(\R\setminus(\{2b\N+\la\}\cup\{2b\N\}))$. As in
\cite[Lemma 8]{JBVBGR}, we use the Helffer-Sj\"ostrand formula and
we have
$$\frac{d}{d\ep}f(H_0+\ep
V)|_{\ep=0}=\frac1\pi\int_\C \overline{\p}\Ti
f(z)(H_{0}-z)^{-1}V(H_{0}-z)^{-1} L(dz),$$ for $\Ti f\in
C_0^\infty(\R^2)$ an almost analytic extension of $f$, (i.e. $\Ti
f_{|\R}=f$ and $\overline{\p}_\la\Ti f(\la)=O(|\mbox{Im}\la|^\infty)$)
and $L(dz)$ denotes the Lebesgue measure on $\C$.
\medskip

Let us now define $$\si_\pm(z)=\tr((H_{0}-z)^{-1}V(H_{0}-z)^{-1}),\
\ \ \pm\im(z)>0.$$ The functions $\si_\pm(z)$ satisfy the relation
\begin{eqnarray}\label{eqsigma pm schrodinger}
 \si_-(z)=\overline{\si_+(\overline{z})},\ \ \ \im(z)<0.
\end{eqnarray}
For $\te\in\R$, the operator
$$(H_{0}-z)^{-1}V(H_{0}-z)^{-1},$$
is unitarly equivalent to the operator
$$(H_{0,\te}-z)^{-1}V_\te(H_{0,\te}-z)^{-1}.$$
Using the cyclicity of the trace, we deduce
\begin{eqnarray}\label{eqsigmapm}\si_\pm(z)=\tr(\p_z{
T}_{V,\te}(z)
) ,\ \ \pm\im(z)>0,\ \te\in\R.\end{eqnarray}
From Lemma \ref{lemTVholomorphe z in Om_q}, the function $\te\longrightarrow\p_z{
 T}_{V,\te}(z)$ is holomorphic on $D_\epsilon^+$ with value in the trace class for $\im(z)>0$. 
 Then, (\ref{eqsigmapm}) is also available for $\te\in D_\epsilon^+$ and taking $\im\te>0$, $z\longrightarrow\si_+(z)$ can be extended to $\R\setminus(\{2b\N+\la\}\cup\{2b\N\})$. According to (\ref{eqsigma pm schrodinger}), $\si_-(z)$ satisfies the same property of $\si_+(z)$.
\medskip

Hence, $\frac{d}{d\ep}f(H_0+\ep V)|_{\ep=0}$ is of trace class, and
$$\tr\left(\frac{d}{d\ep}f(H_0+\ep
V)|_{\ep=0}\right)=\frac1\pi\int_{\im(z)>0} \overline{\p}\Ti
f(z)\si_+(z) L(dz)+\frac1\pi\int_{\im(z)<0} \overline{\p}\Ti
f(z)\si_-(z) L(dz).$$ Then the Green formula yields the lemma.
\end{proof}
\medskip

We will deduce Theorem \ref{thm Breit-Wigner} from Theorem \ref{thm
representation formula} by using the previous lemma and the
cyclicity of the trace.
\bigskip

\emph{Acknowledgments}. The author is grateful to V. Bruneau and
J.-F. Bony for many helpful discussions. We also thank the French
ANR (Grant no. JC0546063) for the financial support.

\end{document}